\documentclass[12pt]{amsart}
\usepackage{amsmath}
\usepackage{amsthm,amssymb}
\usepackage{amssymb}
\usepackage{graphicx}
\usepackage[utf8]{inputenc}
\usepackage{color}
\usepackage{enumerate}

\theoremstyle{plain}
\newtheorem{teo}{Theorem}[section]
\newtheorem{thm}[teo]{Theorem}
\newtheorem{lem}[teo]{Lemma}

\newtheorem{defn}[teo]{Definition}

\def\out-deg{\text{\rm Out-deg}}

\usepackage[margin=1in,letterpaper,portrait]{geometry}
\usepackage{hyperref}
\usepackage{comment}
\newtheorem{cor}[teo]{Corollary}
\newcommand{\id}{\mathrm{id}}
\newcommand{\reverse}{\pi}

\newcommand{\hs}{\hat{s}}

\begin{document}

\title[Two-term  identities]{\textbf{Two-Term Polynomial Identities
in Algebras}}
\date{}

\author{Allan Berele}
\address[A.~Berele]{Department of Mathematical Sciences, DePaul University, Chicago, IL, United States}
\email{aberele@depaul.edu}

\author{Peter Danchev}
\address[P.~Danchev]{Institute of Mathematics and Informatics, Bulgarian Academy of Sciences, 1113 Sofia, Bulgaria}
\email{danchev@math.bas.bg; pvdanchev@yahoo.com}
\thanks{P.~Danchev was supported in part by the Junta de Andaluc\'ia under Grant FQM 264. }

\author{Bridget Eileen Tenner}
\address[B.~E.~Tenner]{Department of Mathematical Sciences, DePaul University, Chicago, IL, United States}
\email{bridget@math.depaul.edu}
\thanks{B.~E.~Tenner was partially supported by NSF Grant DMS-2054436.}

\subjclass[2010]{16R10, 16R40, 16U80}

\keywords{commutativity, generalized commutativity, eventual commutativity}

\begin{abstract}
\noindent We study algebras satisfying a two-term multilinear identity, namely one of the form $x_1\cdots x_n=qx_{\sigma(1)}\cdots x_{\sigma(n)}$, where $q$ is a parameter from the base field. We show that such algebras with $q=1$ and $\sigma$ not fixing~1 or~$n$ are eventually commutative in the sense that the equality $x_1\cdots x_k =x_{\tau(1)}\cdots x_{\tau(k)}$ holds for $k$ large enough and all permutations $\tau\in S_k$. Calling the minimal such~$k$ the degree of eventual commutativity, we prove that $k$ is never more than~$2n-3,$
and that this bound is sharp.  For various natural examples, we prove that $k$ can be taken to be $n+1$ or $n+2$. In the case when $q\ne1$, we establish that the algebra must be nilpotent.

We, moreover, demonstrate that if an algebra is eventually commutative of arbitrary characteristic, then it has a finite basis of its polynomial identities, thus confirming the Specht conjecture in this particular case.
\end{abstract}

\maketitle

\section{Introduction and motivation}

The definition of a \emph{commutative algebra} (over an arbitrary field) is well-known and studied early in a mathematician's career: for any two elements $a,b$ in an algebra $A$, the equality $ab=ba$ must be true; that is, the polynomial identity $xy=yx $  is fulfilled in $A$. Despite the ease and familiarity of this definition, it is sometimes not easy to decide whether a given algebra is commutative or not. A recent article relevant to this question is \cite{BD}, and the reader is referred to any of the sources \cite{AGPR}, \cite{DF}, \cite{GZ}, or \cite{R} for the general theory of polynomial identities of algebras.

A weaker form of commutativity is ``eventual commutativity.'' We say that an algebra $A$ is \emph{eventually
commutative of degree~$k$} if it satisfies the identity $$x_1\cdots x_k=x_{\tau(1)}\cdots x_{\tau(k)}$$
for all permutations $\tau\in S_k$.  Of course, if $A$ has a unit, then eventual commutativity is the same as commutativity; thus, we will {\it not} assume that $A$ has a unit. It is not hard to see that if $A$ is eventually commutative of degree~$k$, then $A$ will also be eventually commutative of all higher degrees. So, we will say that $A$ is \emph{eventually commutative} to mean that is eventually commutative of some unspecified degree. Note
that in an eventually commutative ring the commutator ideal must be nilpotent. More specifically, if an algebra is eventually commutative of degree~$k$, then any product of $\lceil \frac12 k\rceil$ commutators will be zero. It follows that if the algebra is semiprime (i.e., it contains no nilpotent ideals) it will be commutative. On the other hand, the direct sum of a commutative algebra with a nilpotent one would be an example of a noncommutative algebra which is eventually commutative.

In 1969, when the Specht conjecture---that every $T$-ideal is finitely generated as a $T$-ideal---was the most celebrated open problem in p.i.~algebras, Latyshev wrote~\cite{La} in which he proved the conjecture for algebras satisfying two-term identities in characteristic zero. In the course of his proof, he proved a powerful lemma which we will show can be used to study eventual commutativity of these algebras.

Our primary focus will be on algebras satisfying identities of the form
$$x_1\cdots x_n=x_{\sigma(1)}\cdots x_{\sigma(n)},$$
for some~$n$ and for some fixed permutation $\sigma \in S_n$. Our main general result, Theorem~\ref{thm:main}, is that if $A$ satisfies such an identity for a permutation $\sigma$ such that $\sigma(1)\ne1$ and $\sigma(n)\ne n$, then $A$ will be eventually commutative of degree at most $2n-3$, and we show that this bound is best possible
by constructing such an algebra for which the degree of eventual commutativity is exactly $2n-3$.

More generally, if we drop the hypotheses that $\sigma$ fixes neither~1 and nor~$n$, then we can write
$x_{\sigma(1)}\cdots x_{\sigma(n)}$ in the form $x_1\cdots x_i u x_j\cdots x_n$, where $u$ does not begin with $x_{i+1}$ nor end with $x_{j-1}$. In this case, our theorem implies that $A$ satisfies all identities of the form
$$x_1\cdots x_k=x_1\cdots x_i \,x_{\tau(i+1)}\cdots x_{\tau(k)}\,x_{k+1}\cdots x_{n-j+k},$$
where $\tau$ is any permutation of $\{i+1,\ldots,k\}$, $k\ge2(j-i)-3$.

In Section~\ref{sec:special}, we turn to some natural special cases of permutations for which an identity $x_1\cdots x_k = x_{\sigma(1)} \cdots x_{\sigma(k)}$ might be known. Perhaps the most important is what we call \emph{generalized commutativity}, in which $\sigma$ is the permutation reversing the sequence $1,2,\ldots,n$ (often called the \emph{longest permutation} or the \emph{long element}). In this case, if $n\not\equiv 1$ mod~4, then $A$ will be eventually commutative of degree~$n+1$, and otherwise it will be eventually commutative of degree~$n+2$. We next consider the cases in which $\sigma$ is the transposition $(1n)$ or the $n$-cycle $(12\cdots n)$. In both of
these scenarios, we prove that $A$ will be eventually commutative of degree~$n+1$.

We round out the picture in Section~\ref{sec:last} by considering more general two-term polynomial identities, namely
$$x_1\cdots x_n=qx_{\sigma(1)}\cdots x_{\sigma(n)},$$
where $q\ne1$ is some element of the base field. We prove that, an algebra satisfying such an identity will be
nilpotent (see Theorem~\ref{nil}). This is closely related to the classical theorem due to Dubnov-Ivanov-Nagata-Higman which asserts that all nil algebras of bounded degree over a field of characteristic~0 or large enough positive
characteristic are always nilpotent (cf. \cite{AGPR} or \cite{DF}).

Since Latyshev wrote \cite{La}, the Specht conjecture was proven by Kemer in
\cite{K1} in characteristic 0, and counterexamples have been found in characteristic~$p$ (see Belov~\cite{B} and~\cite{Bel}, Grishin~\cite{G} and ~\cite{Gr} and Shchigolev~\cite{Sch}). In our last section, we confirm the Specht conjecture in any characteristic in  eventually commutative algebras. The main ideas of this section were generously contributed by the anonymous expert referee to whom we express our sincere gratitude.

We also owe a debt of gratitude to V.~Drensky for pointing out Latyshev's results to us and sending the paper \cite{La}; and to the anonymous referee for a number of improvements to this paper.

\section{General theorems}\label{sec:2}

\subsection{Definitions}

We assume throughout that $A$ is an algebra satisfying the polynomial identity
\begin{equation}
f_\sigma(x_1,\ldots,x_n)=x_1\cdots x_n-x_{\sigma(1)}
\cdots x_{\sigma(n)},
\label{eq:1}
\end{equation}
where $\sigma\in S_n$ is a permutation, and unless specified otherwise, we assume that $\sigma$ does not fix~1 or~$n$. As is commonplace in p.i.~theory, we conflate the statement that $A$ satisfies $f_\sigma$ with the statement $A$ satisfies $f_\sigma=0$. As mentioned in the introduction, we do not assume that $A$ has a unit element. Also, Kemer's classification of non-matrix varieties in~\cite{K} assures that if the characteristic is~0, then these algebras will be commutative-by-nilpotent, but our result is stronger and works in any characteristic.

For each permutation $\tau\in S_k$, we let $x_\tau $ be the monomial $x_{\tau(1)}\cdots x_{\tau(k)}$, and let $f_\tau=f_\tau(x_1,\ldots,x_k)$ be the polynomial $x_{\id}-x_\tau$. (We are using ``$\id$''  for the
identity permutation, and hoping that context will make it clear which symmetric group it is the identity of.) If $M_k$ is the set of all multilinear monomials of degree $k$ in $x_1,\ldots,x_k$, then the identification $\tau\leftrightarrow x_\tau$ identifies $S_k$ with $M_k$. For each $k$, we let
$$H_k\subseteq S_k$$
be the set of permutations such that $f_\tau$ is an identity for $A$.

\medskip

The following technicality is simple but useful.

\begin{lem}\label{lem:transitive equality}
If an algebra satisfies the identities $f_\tau$ and $f_\nu$, $\tau,\nu\in S_k$, then it satisfies the identity $f_{\tau \nu}$.
\end{lem}

\begin{proof}
Set $y_i := x_{\tau(i)}$, so that $x_\tau=y_1\cdots y_k$. Then,
modulo the identities of $A$,
\begin{align*}
x_1\cdots x_k&=x_{\tau(1)}\cdots
x_{\tau(k)}\\
&=y_1\cdots y_k\\
&=y_{\nu(1)}\cdots y_{\nu(k)}\\
&=x_{\tau\nu(1)}\cdots x_{\tau\nu(k)},
\end{align*}
as required.
\end{proof}

It follows from this, and a little bit more work, that the subset $H_k$ is, in fact, a subgroup.

\begin{cor}\label{cor:subgroup}
The set $H_k$ is a subgroup of $S_k$.
\end{cor}

\begin{proof}
We start by showing that if $f_{\tau}$ is an identity of $A$, then so is $f_{\tau^{-1}}$.
Once again, set $y_i := x_{\tau(i)}$. Hence, and critical for our purposes here, we have $x_i = y_{\tau^{-1}(i)}$. Then,
\begin{align*}
        y_{\tau^{-1}(1)} \cdots y_{\tau^{-1}(k)} &= x_1 \cdots x_k \\
        &= x_{\tau(1)} \cdots x_{\tau(k)}\\
        &= y_1 \cdots y_k.
\end{align*}
Therefore, the set $\{\sigma \in S_k : x_{\id} = x_{\tau}\}$ contains inverses and, thanks to Lemma~\ref{lem:transitive equality}, is closed under the group operation; thus, it is a subgroup, as claimed.
\end{proof}

Note that the identification of $M_{k}$ with $S_{k}$ induces right and left actions of $S_{k}$ on $M_{k}$.
The left action is by substitution $\tau f(x_1,\ldots,x_{k})=f(x_{\tau(1)},\ldots,x_{\tau(k)})$, and the right action
is by place permutation, i.e., by permuting the order of the factors in the monomials, so $y_1\cdots y_{k}\tau=
y_{\tau(1)}\cdots y_{\tau(k)}$. For example, $(12)x_1x_3x_2=x_2x_3x_1$ and $x_1x_3x_2(12)=x_3x_1x_2$. The reader familiar with the theory of co-characters, developed by Regev (see, e.g., \cite{AGPR}), will recognize these actions as restrictions of the usual actions of the group algebra $FS_n$ over a field $F$ on the space of multilinear polynomials.

\medskip

We now have the following claim.

\begin{cor}
If $\tau,\nu\in S_k$, then $x_\tau=x_\nu$ is an identity for $A$ if and only if $\tau H_k=\nu H_k$.
\label{cor:2.3}
\end{cor}

\begin{proof}
The polynomial $x_\tau-x_\nu$ is an identity if and only if $\tau^{-1}(x_\tau-x_\nu)$ is an identity, if and only if
$x_{\id}-x_{\tau^{-1}\nu}$ is an identity, if and only if $\tau^{-1}\nu\in H_k$, as asserted.
\end{proof}

Recalling the conflation between ``satisfying an identity $f_{\sigma}$'' and ``satisfying $f_{\sigma} = 0$,'' we make the following definition.

\begin{defn}
The monomial $x_\tau$ is \emph{equivalent} to $x_\nu$ if the algebra $A$ satisfies $x_\tau-x_\nu$.
\end{defn}

One way to understand the polynomial identities of $A$ is to answer the following question: How do the identities of degree $n+1$ relate to the identities of degree~$n$? Specifically, which $f_\tau\in M_{n+1}$ are consequences of our existing identity $f_\sigma$?

\begin{defn}
For each $i \in [1,n]$, define $T_i(f_\sigma)$ to be
$$T_i(f_\sigma)=f_\sigma(x_1,\ldots,x_ix_{i+1},\ldots,x_{n+1}).$$
Additionally, let $T_0(f_\sigma)$ be $x_1f_\sigma(x_2, \ldots,x_{n+1})$ and let $T_{n+1}(f_\sigma)$ be
$f_\sigma(x_1,\ldots,x_n)x_{n+1}$.
\end{defn}

These polynomials will be of the form $f_\tau$ for some $\tau\in S_{n+1}$, and we describe how $\tau$ can be
computed from $\sigma$ and~$i$. Let $\hs_i$ be the cycle permutation $(i,i+1,\ldots,n+1)$.

\medskip

We now come to the following two establishments.

\begin{lem}
With notation as above, $T_i(f_\sigma)$ equals $f_\tau$, where $\tau = \hs_{i+1}\sigma\hs_{
\sigma^{-1}(i)+1}^{-1}$, identifying $\sigma\in S_n$ with the permutation in $S_{n+1}$ fixing $n+1$.\label{lem:2.4}
\end{lem}

\begin{proof}
The right action of $\hs_{\sigma^{-1}(i)+1}^{-1}$ on $x_\sigma x_{n+1}$ moves $x_{n+1}$ to the position to the right of $x_i$, i.e., if $x_\sigma=ux_ivx_{n+1}$, then $\hs_{\sigma^{-1}(i)+1}^{-1} x_\sigma$ equals $ux_ix_{n+1}v$; and then the left action of $\hs_{i+1}$ increases the indices of all $x_j$ for $n+1>j>i$ and substitutes $x_{i+1}$ for $x_{n+1}$.
\end{proof}

\begin{thm}\label{thm:Consequences} Let $A$ be a p.i.~algebra whose $T$-ideal of identities is generated by the set of $f_\sigma$, for $\sigma\in H_n$. Then, for all $k>n$, the space of multilinear identities of degree~$k$ in
$x_1,\ldots,x_k$ is spanned by the set $\{\nu f_\tau| \tau\in H_k\}$, and all elements of $H_k$ are gotten from $H_n$ by repeated applications of the~$T_i$ to $H_n$.
\end{thm}

\begin{proof}
The space of polynomials which are consequences of $f_\sigma$ is spanned by all
\begin{equation}
u_0f_\sigma(u_1,\ldots,u_n)u_{n+1},\label{eq:consequnce}
\end{equation}
where, because $A$ is not assumed to have a unit, all of the terms in $u_1,\ldots,u_n$ are of degree at least~1. Since $f_\sigma$ is multilinear, we may restrict all the $u_i$ to be monomials. In order for \eqref{eq:consequnce} to be
multilinear and degree~$k$ in $x_1,\ldots,x_k$, the product $u_0\cdots u_{n+1}$ must be a permutation of $x_1\cdots x_k$, in which case \eqref{eq:consequnce} will be of the form $f_\tau-f_\mu$. According to Corollary~\ref{cor:2.3}, it will be a consequence of some $f_\tau$, $\tau\in H_k$.

Finally, the expression~\eqref{eq:consequnce} will be of the form $f_\tau$ if one of $u_0\cdots u_{n+1}$ or $u_0u_{\sigma(1)} \cdots u_{\sigma(n)}u_{n+1}$ equals $x_1\cdots x_k$. In the former case, the permutation $\tau$ can be obtained from~$\sigma$ by repeated applications of the~$T_i$. In the latter case, one finds that
$$-u_0f_{\sigma^{-1}}(u_{\sigma(1)},\ldots,u_{\sigma^{-1}(n})
    u_{n+1}=u_0u_{\sigma(1)}
\cdots u_{\sigma(n)}u_{n+1}-u_0u_1\cdots u_nu_{n+1},$$
which equals \eqref{eq:consequnce}, reducing to the former case, as asked for.
\end{proof}

\subsection{Eventual commutativity}\label{sec:two term identities}

In this section, we prove that any algebra satisfying (\ref{eq:1}), with $\sigma$ fixing neither $1$ nor $n$, will be eventually commutative; namely, it will satisfy
$$f_\tau =x_1\cdots x_k-x_{\tau(1)}\cdots x_{\tau(k)},$$
for all $k\ge 2n-3$ and for all $\tau\in S_k$.

\medskip

In \cite{La} Latyshev proved the Specht conjecture for algebras satisfying two-term identities. In the course of his
proof he proved the following theorem, which turns out to be a key ingredient in our study of eventual commutativity. Although Latyshev's paper assumes a priori characteristic zero, the proof of this theorem is purely combinatorial and does not rely on any characteristic, as we will indicate below.

\begin{thm} If an algebra satisfies $f_\sigma$ for some $\sigma\in S_n$ not fixing~1, then it satisfies the identities
$$x_1\cdots x_{k+n}=x_{\tau(1)}\cdots x_{\tau(k)}x_{k+1}
\cdots x_{k+n}$$ for all $k$ and all $\tau\in S_k$.
\label{thm:Latyshev}
\end{thm}

For future reference and to display the fact that the proof is characteristic independent, we record the main
ideas of Latyshev's proof. Write $x_\sigma=x_iu$. Thus, compute the identity
\begin{equation}\label{eq:Latyshev}
zf_\sigma(x_1,\ldots,x_n)-f_\sigma(x_1,\ldots,zx_i,\ldots,x_n)=
zx_1\cdots  x_n -x_1\cdots zx_i\cdots x_n
\end{equation}
Next, make three substitutions into~\eqref{eq:Latyshev}. For the first one, let $z=y_1$ and
$x_1=y_2x_1$:
$$y_1y_2x_1\cdots x_n=y_2x_1\cdots y_1x_i\cdots x_n.$$
Then, let $z=y_2$ and $x_i=y_1x_i$:
$$y_2x_1\cdots y_1x_i\cdots x_n=x_1\cdots y_2y_1
\cdots x_n.$$
Finally, let $z=y_2y_1$:
$$y_2y_1 x_1\cdots x_n= x_1\cdots y_2y_1x_i\cdots
x_n.$$
These equations imply that
\begin{equation}y_1y_2x_1\cdots x_n=y_2y_1x_1\cdots x_n
\label{eq:y1y2}\end{equation}
which ensures the mentioned theorem.

\medskip

We make some observations about this theorem. First, the conclusion can be restated as the algebra satisfies $f_{\tau}$, where $\tau$ is considered an element of $S_{n+k}$ fixing the elements of $\{k+1,\ldots,k+n\}$.
Second, that essentially the same argument (or by applying the theorem to $A^{op}$, the opposite algebra)
guarantees that if $\sigma\in S_n$ does not fix~$n$, then $f_\tau$ is a consequence of $f_\sigma$
for all $\tau\in S_{n+k}$ fixing the first $n$ letters $\{1,\ldots,n\}$ as well as analogues of the
above numbered equations. Let $x_\sigma=vx_j$, then
\begin{align}
x_1\ldots x_nz&=x_1\ldots x_jz\ldots x_n\tag{\ref{eq:Latyshev}$^\prime$}
\label{eq:Latyshev'}\\
x_1\ldots x_ny_1y_2&=x_1\ldots x_ny_2y_1\tag{\ref{eq:y1y2}$^\prime$}
\label{eq:y1y2'}.
\end{align}
It is not too difficult to prove from this that an algebra satisfying such an equation will be eventually commutative of degree~$2n+1$, but we are aiming for a smaller bound.

\medskip

It will also useful to translate equations~\eqref{eq:Latyshev}
and~\eqref{eq:y1y2} into the language of $H_k$.

\begin{cor} Let $A$ be an algebra satisfying some $f_\sigma$, where $\sigma\in S_n$, $x_\sigma=x_iu=vx_j$, and where $i\ne1$ and $j\ne n$. Then, $H_{n+1}$ contains the cycles $(12\ldots i)$ and $(j+1\ldots n+1)$; and
$H_{n+2}$ contains the transpositions~$(12)$ and $(n+1,n+2)$.
\label{cor:4identities}
\end{cor}

\begin{proof} In \eqref{eq:Latyshev}, we set $z=x_1$ and every $x_i$ to $x_{i+1}$ to get $(12\ldots i)\in H_{n+1}$; and in~\eqref{eq:y1y2} set $y_1=x_1,\  y_2=x_2$ and each $x_i$ to $x_{i+2}$ to get $(12)\in H_{n+2}$. For the other two, we set $z=x_{n+1}$ in~\eqref{eq:Latyshev'} and $y_1=x_{n+1},\ y_2=x_{n+2}$ in \eqref{eq:y1y2'} (or just use $A^{op}$), deducing the claim.
\end{proof}

\begin{lem} With notation as in the above corollary, $H_{n+2}$ contains all permutations of $\{1,2,\ldots,i+1\}$ and of $\{j+1,\ldots,n+2\}$.\label{lem:2.10}
\end{lem}

\begin{proof} Let $\tau=(12\ldots i)\in S_{n+1}$. By virtue of Corollary~\ref{cor:4identities}, $A$ satisfies the identity $f_\tau$, which equals
$$x_{\id}-x_2x_3\cdots x_ix_1x_{i+1}\cdots x_{n+1}.$$
Applying the operator $T_i$ yields the identity
$$f_{\id}-x_2x_3\cdots x_{i+1}x_1x_{i+2}\cdots x_{n+2},$$
which shows that $(12\ldots i+1)$ is in $H_{n+2}$. On the other hand, $H_{n+2}$ contains as well the transposition $(12)$, so Corollary~\ref{cor:subgroup} allows us to deduce that $H_{n+2}$ will contain the subgroup they generate, which is the symmetric group on $\{1,2,\ldots, i+1\}$. The other case is similar, concluding the proof.
\end{proof}

This suggests the following definition.

\begin{defn}
For $n\ge a,b$, let $S(n;a,b)$ be the subgroup of $S_n$ generated by the permutations of the first~$a$ letters and the permutations of the last~$b$ letters.
\end{defn}

Note that, if $a+b>n$, then $S(n;a,b)=S_n$ and otherwise $S(n;a,b)$ is a proper subgroup.

\begin{lem} If the transposition $(ij)$ is in $H_k$, then the transpositions $(ij)$ and $(i+1,j+1)$ are in $H_{k+1}$.\label{lem:2.13}
\end{lem}

\begin{proof} These correspond to $f_{(ij)}(x_1,\ldots,x_k)x_{k+1}$ and $x_1f_{(ij)}(x_2,\ldots,x_{n+1})$, respectively, which concludes the proof.
\end{proof}

\begin{lem} If $S(k;a,b)\subseteq H_k$, then $S(n+1;a+1,b+1)\subseteq
H_{k+1}$\label{lem:increasing n}
\end{lem}

\begin{proof} One sees that $H_k$ contains all of the transpositions $(i,i+1)$, where either $i+1\le a$ or $i\ge n-j+1$. With the previous lemma at hand, $H_{k+1}$ will contain all $(ij)$, where either $i+1\le a+1$ or $i\ge n-j+1$.  These, however, generate the subgroups of all permutations of the first $a+1$ letters and of the last $b+1$ letters, respectively, as needed.
\end{proof}

The following is the last main ingredient of our proof.

\begin{lem} If $A$ satisfies $f_\sigma$, where $\sigma\in S_n$ fixes neither~1 nor $n$, then $H_{n+2}$ contains $S(n+2;4,4)$.
\label{lem:S44}
\end{lem}

\begin{proof} We concentrate on the proof that, if $\sigma(1)\ne 1$, then $H_{n+4}$ contains all permutations of $\{1,2,3,4\}$; the proof that it contains the permutations of the last four letters is similar. There are three
cases to consider.

If $\sigma(1)\ge 3$, then the lemma follows from Lemma~\ref{lem:2.10}. Since $\sigma(1)\ne 1$, we henceforth assume $\sigma(1)=2$.

If $\sigma^2(1)=\sigma(2)\ne 1$, then $\sigma^2(1)\ge 3$. Since $f_{\sigma^2}$ is an identity, we are again done by Lemma~\ref{lem:2.10}.

Finally, we consider the case of $\sigma(1)=2$ and $\sigma(2)=1$. Note that Corollary~\ref{cor:4identities} reaches us that $(12)\in H_{n+1}$. Let $x_\sigma=x_2x_1x_iu$, so $f_\sigma=x_{\id}-x_2x_1x_iu$. Following in the footsteps of Latytyshev, we compute
$$f_\sigma(x_1z,x_3,\ldots,x_{n})-f_\sigma(x_1,\ldots,zx_i,
\ldots,x_{n})=x_1yx_2\cdots x_n-x_1\cdots zx_i\cdots x_n.$$
We now substitute $x_2$ for $z$ and $x_{j+1}$ for $x_j$, for all $j\ge 2$. The result is
$$x_{\id}-x_1 x_3x_4\cdots x_ix_2x_{i+1}\cdots x_{n+1}=f_\tau,$$
where $\tau=(23\ldots i)$. Hence, $H_{n+1}$ contains the cycle $\tau=(23\ldots i)$. However, since it contains $(12)$, it follows that it contains the product $(12)\tau=(12\ldots i)$ too. Since $(12)$ and $(12\ldots i)$ generate the permutation group $S_i$, all permutations of $\{1,2,\ldots,i\}$ are in $H_{n+1}$, and since $i\ge 3$ the
arguments follow from Lemma~\ref{lem:2.13}.
\end{proof}

We now can attack our chief result in this section.

\begin{thm} If an algebra $A$ satisfies $f_\sigma$ for some $\sigma\in S_n$ fixing neither~1 nor~$n$, then
if $n\ge 5$, the algebra $A$ is eventually commutative of degree $2n-3$, and if $n=3$ or $n=4$, then $A$ is eventually commutative of degree~$n+1$.
\label{thm:main}
\end{thm}

\begin{proof}
We leave the cases of $n=3$ or $n=4$ to the reader, remarking that the material in Section~\ref{sec:special} can be used as a shortcut.

Furthermore, appealing to Lemma \ref{lem:S44}, it follows that $S(n+2;4,4)\subseteq H_{n+2}$. Then, by iterated use of Lemma~\ref{lem:increasing n}, one observes that $S(n+2+b;4+b)\subseteq H_{n+2+b}$. Finally, taking $b=n-5$ gives
$S(2n-3;n-1,n-1)\subseteq H_{2n-3}$, but since $n-1+n-1>2n-3$, we derive $S(2n-3;n-1,n-1)=S_{2n-3}$, as desired.
\end{proof}

We now show that the bound of $2n-3$ is sharp by exhibiting a permutation $\sigma\in S_n$ which implies eventual commutativity of degree $2n-3$, but of no lower degree. To shorten notation, we denote $S(n;a,a)$ as $S(n;a)$.

\begin{lem}
Let $\sigma\in S(n;a)$ and $n\ge 2a+1$. Then, $T_i(\sigma)$ belongs to $S(n+1;a+1)$, for all~$i$.
\end{lem}

\begin{proof}
We first show that, if $k\le a+1$, then $T_i(\sigma)(k)\le a+1$. So, Lemma~\ref{lem:2.4} applies to get that
$$T_i(\sigma)(k)=\hs_{i+1}\sigma\hs_{j+1}^{-1}(k),$$
where $j=\sigma^{-1}(i)$. Thus, $\hs_{j+1}^{-1}(k)$ will be either $k,\ k-1$ or $n+1$. In the former two cases,
$\sigma\hs_{j+1}^{-1}(k)$ will be less than or equal to $a$, and so $\hs_{i+1}$ acting on $\sigma\hs_{j+1}^{-1}(k)$ will be less than or equal to $a+1$. In the case of $\hs_{j+1}^{-1}(k)=n+1$, we have $k=j+1$, and so $j\le a$ and, likewise, $i=\sigma(j)\le a$. Therefore, $\hs_{i+1}\sigma(n+1)=\hs_{i+1}(n+1)=i+1\le a+1$. The case
of $k\ge n-a+1$ is similar.

Furthermore, we need to show that, if $a+2\le k\le n-a$, then $T_i(\sigma)(k)$ is equal to $k$. There are three possible cases. First, consider $i\le a$. Then, $j$ is also less than or equal to $a$, and $i+1$ and $j+1$ are less than or equal to $a+1$, and so are less than~$k$. Hence, $\hs_{j+1}^{-1}(k)=k-1$, and $\sigma(k-1)=k-1$, since $k-1$ is between $a+1$ and $n-a-1$. Therefore, $\hs_{i+1}(k-1) =k$. The second case is similar: if $a+1\le i\le n-a$, then
$\sigma(i)=i$, so $\sigma^{-1}(i)=i$ and $T_i(\sigma)= \hs_{i+1}\sigma \hs_{i+1}^{-1}$. Thus, $\hs_{i+1}^{-1}(k)$
will equal $k$, or $k-1$, or $n+1$. In all of these cases, $\sigma \hs_{i+1}^{-1}(k)=\hs_{i+1}^{-1}(k)$ and so
$\hs_{i+1}$ of it will be~$k$. Finally, if $i\ge n-a+1$, then $j\ge n-a+1$ and all three of $\hs_{i+1}$, $\sigma$ and
$\hs_{j+1}^{-1}$ fix $k$, giving the wanted.
\end{proof}

We are now ready to establish the following.

\begin{thm}
Set $\sigma=(12)(n-1,n)\in S_n$ for $n\ge4$. Then, an algebra satisfying $f_\sigma$ will be eventually commutative
of degree $2n-3$, but need not be eventually commutative of any lesser degree.
\end{thm}

\begin{proof}
In the proof that follows, we will assume that $n\ge5$. Indeed, if $n=4$, a direct computation of the $T_i(f_\sigma)$ demonstrates that $H_5=S_5$; we leave those details to the interested reader for inspection.

That the identity $f_\sigma$ does not imply eventual commutativity of degree less than $2n-3$ follows from the
previous lemma: in fact, by induction, all consequences $f_\tau$ of $f_\sigma$ of degree $n+b$, for $b\le n-4$, will lies in $S(n+b;2+b)$, a proper subgroup of $S_{n+b}$.
\end{proof}

\section{Special cases}\label{sec:special}

\subsection{Generalized commutativity}\label{sec:general and eventual}

The main notion in this section is the following variation of commutativity in an algebra.

\begin{defn}\label{defn:generalized and eventual}
\label{defn:long element}
Fix a positive integer $n$. Define $\pi_n \in S_n$ to be the permutation given by $\pi(i) = n+1 - i$ for all $i$. That is, this $\pi_n$ reverses the sequence $12\cdots (n-1)n$.
An algebra $A$ is \emph{generalized commutative of degree $k$} if it satisfies the identity
$$x_1\cdots x_k=x_k\cdots x_1=x_{\pi(n)}.$$
\end{defn}

The permutation $\pi_n$ is what is known as the \emph{long element} in the context of Coxeter groups (see, for example, \cite{BB}).

\begin{lem}\label{lem:general commutativity and even permutations}
If an algebra is generalized commutative of degree $n$, then it satisfies the identities $f_\tau$ for all even permutations $\tau\in S_{n+1}$.
\end{lem}

\begin{proof}
Let $s_i$ be the adjacent transposition $(i,i+1)$. These permutations, also called \emph{simple reflections}, generate the symmetric group, and the collection of pairwise products $\{s_is_j\}$ generates the alternating group (see, for instance, \cite{BB}).

Let $\reverse:= \reverse_{n+1}$ be as in Definition~\ref{defn:long element}. Notice that $\reverse^2 = \id$, and $s_i\reverse=\reverse s_{n+1-i}$ for all $i \in [1,n]$.

Because $A$ satisfies generalized commutativity of degree $n$, we have
$$x_1 x_2 \cdots (x_i x_{i+1}) \cdots x_n x_{n+1} = x_{n+1} x_n \cdots (x_i x_{i+1}) \cdots x_2 x_1.$$
In other words, $x_{\id} = x_{s_i \reverse}$ for all $i \in [1,n]$. Referring to Lemma~\ref{lem:transitive equality}, it follows that
$$x_{\id} = x_{(s_i\reverse)(s_h\reverse)}$$
for all $i, h \in [1,n]$.
Recall that $(s_i\reverse)(s_{n+1-j}\reverse) = s_i s_j\reverse^2=s_i s_j$. Thus, $x_{\id} = x_{s_is_{n+i-h}}$ for all $i, h \in [1,n]$, and setting $j:=n+i-h$ shows that ${s_is_j}\in H_{n+1}$ for all $i, j \in [1,n]$. The result follows now from Lemma~\ref{lem:transitive equality}.
\end{proof}

For use in the next section, we state the following assertion in a slightly greater generality than needed for our present purposes.

\begin{lem}\label{lem:if one odd}
Assume that an algebra $A$ satisfies $f_\tau$ for all even $\sigma\in S_k$. If $A$ also satisfies $f_\nu$ for at least one odd permutation $\nu\in S_k$, then $A$ is eventually commutative of degree $k$; otherwise, it is eventually commutative of degree $k+1$.
\end{lem}

\begin{proof}
The result is trivial when $k=2$, so assume that $k \ge 3$.

Furthermore, Corollary~\ref{cor:subgroup} shows that $H_k$ is a group and that, by hypothesis, it contains the alternating group $A_k$, which is of index $2$. Hence, if $H_k$ is bigger than $A_k$, then $H_k$ must be all of $S_k$.

It is, however, not hard to see that, if $H_k$ contains $A_k$, then $H_{k+1}$ contains $A_{k+1}$.

Now, suppose that $x_1\cdots x_k$ is not equivalent to any odd permutation of the $x_i$. The permutation $(123)$, written here in cycle notation, is even. Thus, in degree $k$, the algebra $A$ satisfies the identity $f_{(123)}$:
$$x_1\cdots x_k=x_2x_3x_1x_4\cdots x_k;$$
then, in degree $k+1$, it will satisfy $T_2(f_{(123)})$, namely
$$x_1(x_2x_3)x_4x_5\cdots x_{k+1} = (x_2x_3)x_4x_1x_5\cdots x_{k+1},$$
which insures that $(1234)\in H_{k+1}$. But this permutation $(1234)$ is odd, completing the proof.
\end{proof}

We now arrive at the principal result of this section.

\begin{thm}\label{thm:eventually commutative}
Let $A$ be generalized commutative of degree $n$. If $n\not\equiv 1$  \text{mod} $4$, then $A$ is eventually commutative of degree $n+1$; otherwise, it is eventually commutative of degree $n+2$.
\end{thm}

\begin{proof}
In view of Lemma~\ref{lem:general commutativity and even permutations}, we have $f_\tau\in H_{n+1}$ for all even permutations $\tau \in S_{n+1}$. The length of $\reverse_k \in S_k$ is $k(k-1)/2$, so $\reverse_k$ is odd if $k\equiv 2$ or $3$, and even if $k\equiv 0$ or $1$, all modulo $4$.

As noted in the proof of Lemma~\ref{lem:general commutativity and even permutations}, the algebra $A$ satisfies the identities $f_{s_i\reverse_{n+1}}$. If $n+1\equiv 0$ or $1$ (i.e., if $n \equiv 0$ or $3$), then $s_i\reverse_{n+1}$ is odd. Consequently, owing to Lemma~\ref{lem:if one odd}, the algebra $A$ is eventually commutative of degree $n+1$.

However, the algebra $A$ always satisfies
$$x_1\cdots x_{n+1}=x_n\cdots x_1 x_{n+1},$$
which can be written as $x_{\id}=x_{\rho}$, where $\rho$ has the same parity as $\reverse_n$. Hence, this $\rho$ is odd if $n\equiv 2$ or $3$, and this case is again handled using Lemma~\ref{lem:if one odd}. In sum, the algebra $A$ is eventually commutative of degree $n+1$, provided $n\equiv0,2$ or $3$.

The remaining case when $n\equiv1$ follows from Lemma~\ref{lem:if one odd}, in which case $A$ is eventually commutative of degree $n+2$.
\end{proof}

\subsection{Two more examples}\label{sec:more examples}

We now study two other fairly natural examples. Firstly, consider $\sigma=(1n)$ so that the algebra $A$ satisfies
\begin{equation}
f_{(1n)}=x_1x_2\cdots x_{n-1}x_n=x_nx_2\cdots x_{n-1}x_1.\label{eq:2}
\end{equation}

Thereby, we obtain:

\begin{thm}
If an algebra satisfies \eqref{eq:2}, then it is eventually commutative of degree $n+1$.
\end{thm}

\begin{proof}
In degree $n+1$, the algebra will satisfy $f_\sigma\cdot  x_{n+1}$, so that $(1n)\in H_{n+1}$, and by
Corollary~\ref{cor:4identities} it will contain the cycles $(12\ldots n)$ and $(n,n+1)$, and hence their product $(12\ldots n+1)$. It is well known that $(12)$ and $(12\cdots n+1)$ generate $S_{n+1}$ whence, by conjugation, so do $(n,n+1)$ and $(12\ldots n+1)$, as expected.
\end{proof}

The next case we now consider is $\sigma=(12\cdots n)$, so that
\begin{equation}
f_\sigma=x_1\cdots x_n-x_2\cdots x_n x_1
\label{eq:3a}\end{equation}

Thereby, we receive:

\begin{thm}
If an algebra satisfies \eqref{eq:3a}, then it is eventually commutative of degree $n+1$.
\end{thm}

\begin{proof}
In virtue of Corollary~\ref{cor:subgroup}, an algebra satisfying $f_\sigma$ must also satisfy
$$f_{\sigma^{-1}}(x_1,\ldots,x_n)=x_1\cdots x_n-
x_nx_1\cdots x_{n-1}.$$
Hence, exploiting Corollary \ref{cor:4identities}, $H_{n+1}$ will also contain both $(12)$ and $(12\ldots n+1)$, as promised.
\end{proof}

\section{General two-term identities\label{sec:last}}

\subsection{General permutations}

We now consider the case of permutations not satisfying $\sigma(1)\ne1$ and $\sigma(n)\ne n$. In fact, these more general permutations were considered by Latyshev in \cite{La}, and it is only for expositional reasons that we have been restricting to permutations fixing neither~1 nor $n$ up to this point.

\medskip

We are now prepared to prove the following.

\begin{thm} Let $x_\sigma=x_1\cdots x_i x_{\sigma(i+1)}\cdots x_{\sigma(j)}x_{j+1}\cdots x_n$, where $\sigma(i+1)\ne i+1$ and $\sigma(j) \ne j$. Then, an algebra satisfying $f_\sigma$ will satisfy the identities $f_\tau$, where
$\tau$ is a permutation of $k\ge 2(j-i)+1$ fixing the first $i$ and the last $n-j$ letters.\label{thm:general sigma}
\end{thm}

\begin{proof} As we mentioned, the theorem could be proven using Latyshev's lemma alluded to above. It could also be proven using Lewin's theorem from~\cite{L}. However, it is easy enough to prove directly from Theorem~\ref{thm:main}, which we now do.

We write the term $x_\sigma$ occurring in~\eqref{eq:3} as
$$x_1\cdots x_i u x_{j+1}\cdots x_n,$$
where $u$ has degree at least~2, and $u$ does not start with  $x_{i+1}$ nor end with $x_{j}$. Writing $y_1,\ldots,
y_{j-i}$ for $x_{i+1},\ldots, x_{j}$, equation~\eqref{eq:1} can be written as
$$x_1\cdots x_n- x_1\cdots x_i y_{\tau(1)}\cdots y_{\tau(j-i)}x_j\cdots x_n=x_1\cdots x_i f_{\tau,1}(y_1,\ldots,
y_{j-i})x_j\cdots x_n,$$
for some $\tau\in S_{j-i-1}$. Note that $\tau$ fixes neither~1 nor~$j-i-1$. However, in view of Theorem~\ref{thm:main}, the identity $f_{\tau}$ yields eventual commutativity of degree~$2(j-i)-3$. This means that $y_1\cdots y_{j-i}$ is in the $T$-ideal generated by $f_{\tau}$. Hence,
$$y_1\cdots y_k=\sum u_{\alpha,0}f_{\tau^t}(u_{\alpha,1},\ldots,u_{\alpha,j-i})u_{\alpha,j-i+1},$$
where the $u_{\alpha,i}$ are in the free ring generated by $y_1,\ldots,y_{j-i}$, and all are restricted to have positive degree, except for $\alpha=0$ and $\alpha=j-i+1$. Now, we multiply on the left by $x_1\cdots x_i$ and on the right by $x_{j+1}\cdots x_k$ to get $x_1\cdots x_k$ as a linear combination of terms of the form
\begin{multline*}
x_1\cdots x_i u_{\alpha,0}f_{\tau^t}(u_{\alpha,1},
\ldots,u_{\alpha,j-i})u_{j-i} x_{j+1}\cdots x_k=\\
f_{\sigma^t}(x_1,\ldots,x_{i-1}u_{\alpha,0},
u_{\alpha,1},\ldots,u_{\alpha,j-i}x_{j+1},\ldots,x_k).
\end{multline*}
Consequently, $x_1\cdots x_k$ is a consequence of $f_{\sigma}$, and thus we are done.
\end{proof}

\subsection{The case of $q\ne 1$}

In this section, we consider identities of the form
\begin{equation}
x_1\cdots x_n=q x_{\sigma(1)}\cdots x_{\sigma(n)},\ q\ne1,
\label{eq:3}\end{equation}
where $q$ is an element of the base field, and $\sigma\in S_n$. Specializing all of the $x_i$ to a single variable~$x$ yields $x^n=qx^n$ whence $(1-q)x^n=0$ and $x^n=0$. In characteristic~0 or characteristic greater than~$n$, the Dubnov-Ivanov-Nagata-Higman theorem, see Chapter~6 of~\cite{DF} or Theorem~12.2.16 of~\cite{AGPR}, states that an algebra nil of bounded degree must be nilpotent, but in positive characteristic it need only be locally nilpotent.

Let $f=x_{\id}-qx_\sigma$. The proof of Latyshev's theorem, namely Theorem~\ref{thm:Latyshev}, now
obtains {\it eo ipso\/}, because if $x_\sigma= x_iu$, then
$$zf(x_1,\ldots,x_n)-f(x_1,\ldots,zx_i,\ldots,x_n)=
zx_1\cdots  x_n -x_1\cdots zx_i\cdot x_n,$$
just as in \eqref{eq:Latyshev}. This gives the counterpart of Theorem~\ref{thm:main}, which we state as the following lemma.

\begin{lem} An algebra satisfying \eqref{eq:3}, where $\sigma$ fixes neither~1 nor~$n$, must be eventually commutative.
\end{lem}

With the same proof as in the preceding section, we get this counterpart to Theorem~\ref{thm:general sigma}.

\begin{lem} Let $x_\sigma=x_1\cdots x_i x_{\sigma(i+1)}\cdots x_{\sigma(j)}x_{j+1}\cdots x_n$, where $\sigma(i+1)\ne i+1$ and $\sigma(j) \ne j$. Then, an algebra satisfying \eqref{eq:3} will satisfy the identities $f_\tau$, where
$\tau$ is a permutation of $k\ge 2(j-i)+1$ fixing the first $i$ and the last $n-j$ letters.
\end{lem}

Here is the main theorem of this section.

\begin{thm}\label{nil}
If an algebra $A$ satisfies \eqref{eq:3}, then it is nilpotent.
\end{thm}

\begin{proof}
Right multiplying \eqref{eq:3} by $x_{n+1}\cdots x_{k}$ shows that $A$ satisfies $$x_1\cdots x_{k}  =q x_\sigma x_{n+1}\cdots x_{k}.$$
But, if $k$ is large enough, then by eventually commutativity, $A$ satisfies
$$x_1\cdots x_k = x_\sigma x_{n+1}\cdots x_{k},$$
and so $A$ will satisfy
$$(1-q) x_\sigma x_{n+1}\cdots x_{k+1}=0,$$
yielding $A^{k}=0$, as required.
\end{proof}

\section{Finite Generation}

In this section, we prove that the $T$-ideal of identities of an eventually commutative algebra, in any
characteristic, is finitely generated as a $T$-ideal, which proves the Specht conjecture in this special  case. The main ingredient in the proof is Cohen's theorem about $S$-ideals from~\cite{C}, that we will explicitly state below.

Let $X$ be the set $\{x_1,x_2,\ldots\}$ and let $F\langle X\rangle$ be the free $F$-algebra on $X$. In general, an ideal $I\lhd F\langle X\rangle$ is a $T$-ideal if it is invariant under all homomorphisms $F\langle X
\rangle\rightarrow F\langle X\rangle$. This means that if $f(x_1,\ldots,x_n)$ is in $I$, so is $f(g_1,\ldots,g_n)$
for every $g_1,\ldots,g_n\in F\langle X\rangle$. These $T$-ideals are precisely the ideals of p.i. algebras with~1, and they are sometimes called {\it $T_1$-ideals}. If, however, we want to study algebras that do {\bf not} necessarily have a unit, such as we do in this paper, the corresponding ideal of identities is sometimes called a {\it $T_0$-ideal}, and it is only invariant under substitutions $x_i\mapsto g_i$, where the elements $g_i$ have 0 constant terms. 

We shall say that an ideal in $F\langle X\rangle$ is generated by $f_1,\ldots,f_n$ as a $T_1$-ideal or as a
$T_0$-ideal if it is the smallest such containing $f_1,\ldots,f_n$. For the rest of this section, we will use the term $T$-ideal to refer to $T_0$-ideals and assume $F\langle X\rangle$ is without unit. Our main goal will be to prove that the $T$-ideal of identities of any eventually
commutative algebra is finitely generated.

Cohen's theorem deals with ideals in the commutative algebra $F[X]$ with a weaker invariance property. An ideal $I\lhd F[X]$ is called an {\it $S$-ideal} if, for all polynomials $f(x_1,\ldots,x_n)\in I$ and all order preserving maps $\sigma:\mathbb{N}\rightarrow\mathbb{N}$, the polynomial $f(x_{\sigma(1)},\ldots,x_{\sigma(n)})$ is also in~$I$. For convenience, we will denote this latter polynomial as $f(x_\sigma)$.

\medskip

In~\cite{C}, Cohen proved the following.

\begin{thm}\label{Coh} For every Noetherian ring $F$, the $S$-ideals of $F[X]$ have the Noetherian property. Equivalently, every $S$-ideal is finitely generated as an $S$-ideal.
\end{thm}

We will only be interested in the case in which $F$ is a field.

Let $C\lhd F\langle X\rangle$ be the commutator ideal, and let $C_n$ denote the vector space of homogeneous elements of degree~$n$. Since $C$ is a homogeneous ideal, we have $C=\oplus_n C_n$. We will use the notation $C_{\ge n}$ for the sum $\oplus_{k\ge n}C_k$ and the notation $C_{\le n}$ for the sum $\oplus_{k\le n}C_k$.

\medskip

Before proving the main statement, we need to establish three preliminary technicalities as follows. 

\begin{lem} The $S$-ideal of $F\langle X\rangle$ generated by $C_n\cap F\langle x_1,\ldots,x_n\rangle$ is $C_{\ge n}$.\label{lem:5.2}
\end{lem}

\begin{proof} By the well-known Jacobi identity $[xy,z]=x[y,z]+[x,z]y$, the commutator ideal $C$ is generated
by commutators of the form $[x_i,x_j]$, and so it is spanned by elements of the form $u[x_i,x_j]$, $[x_i,x_j]v$ and  $u[x_i,x_j]v$.
\end{proof}

Let $I$ be $T$-ideal in $F\langle X\rangle$ containing $C_n$. Such ideals are precisely the ideals of identities of algebras, eventually commutative of degree~$n$. Let $\pi:F\langle X\rangle\rightarrow F[X]$ be the usual
projection map. Then, $\pi(I)$ is an $S$-ideal of $F[X]$ and so it has a finite generating set $\pi(f_1),\ldots,\pi(f_t)$, where $f_1,\ldots,f_t\in I$.

\begin{lem} For any $g\in I$ there exists $f$ in the $T$-ideal generated by $f_1,\ldots,f_t$ and by $C_n$ such that $f-g$ has degree less than~$n$.\label{lem:5.3}\end{lem}

\begin{proof} It follows from hypothesis that $\pi(g)$ is in the $S$-ideal generated by $\pi(f_1),\ldots,\pi(f_t)$, say
$$\pi(g)=\sum_{i=1}^t\sum_\sigma a_{i,\sigma}f_i(x_\sigma).$$
Choosing $b_{i,\sigma}$ in $F\langle X\rangle$ such that $\pi(b_{i,\sigma})
=a_{i,\sigma}$, we have
$$\pi(g)=\pi\left(\sum_{i=1}^t\sum_\sigma a_{i,\sigma}f_i(x_\sigma)\right),$$
and so the difference $g-\sum_{i=1}^t\sum_\sigma a_{i,\sigma}f_i(x_\sigma)$
is in the kernel of $\pi$, namely in the commutator ideal~$C$. Call this difference $c$. Since $C$ is a graded ideal, we can decompose $c$ as $c=c_0+c_1$, where $c_0\in C_{\le n-1}$ involves only terms of degree less than or equal to
$n-1$ and $c_1\in C_{\ge n}$ involves the higher degree terms. But $c_1\in I$, and we may take $f=-c_1+\sum_{i=1}^t\sum_\sigma a_{i,\sigma}f_i(x_\sigma),$ as required.
\end{proof}

To complete the proof that $I$ is finitely generated as a $T$-ideal, we need to show that there is a finite set of polynomials that generated the elements of $I$ of degree less than~$n$. Let $I_0$ be the space of such elements.

\begin{lem} The ideal $I_0$ is generated as a $T$-ideal by $I_0\cap F\langle
x_1,\ldots,x_{n}\rangle$.\label{lem:5.4}
\end{lem}

\begin{proof} Let $f\in I$. For any variable $x_i$, we can break $f$ into a sum $f=g+h$, where the monomials occurring in $g$ involve $x_i$ and those in $h$ do not. Hence, setting $x_i=0$, we get that $h$ is in $I$ and so $g$ is also.  Moreover, in $I_0$ each monomial is of degree at most $n-1$ and so involves at most $n$ variables. Since a polynomial of $f(x_{i_1},\ldots,x_{i_{n}})$ can be gotten from $f(x_1,\ldots,x_{n-1})$
using the substitution $x_a\mapsto x_{i_a}$, the lemma follows.
\end{proof}

We can now prove our last theorem, which is the main result of this section.

\begin{thm} If $A$ is an eventually commutative algebra over any field,
then the ideal of identities of $A$ is finitely generated as a $T$-ideal.
\end{thm}

\begin{proof} Let $I\lhd F\langle X\rangle$ be the ideal of identities of~$A$. As in Cohen's Theorem~\ref{Coh}, we can get a finite set $\pi(f_1),\ldots,\pi(f_t)\in I$ that generate the $S$-ideal of the image of $I$ in $F[X]$. Invoking Lemma~\ref{lem:5.3}, the polynomials $f_1,\ldots,f_t$, together with $C_n$ and $I_0$, generate $I$ as a $T$-ideal.   But, by the Lemmas~\ref{lem:5.2} and~\ref{lem:5.4}, each of $C_n$ and $I_0$ are generated by their intersection with $F\langle x_1,\ldots x_n \rangle$. Since each of them involves only polynomials of degree $\le n$, the two intersections are finite dimensional and, therefore, finitely generated, as expected.
\end{proof}

\end{document}